\documentclass[10pt]{amsart}

\usepackage{amssymb,latexsym, mathtools,tikz}
\usepackage{enumerate}
\usepackage{graphicx}
\usepackage{float}
\usepackage{placeins}
\usepackage{mdframed}
\usepackage{amssymb}
\usepackage{esint}
\usepackage{cool}
\usepackage[all,cmtip]{xy}
\usepackage{mathtools}
\usepackage{amstext} 
\usepackage{array}   
\usepackage[shortlabels]{enumitem}
\usepackage{ytableau}
\usepackage{tikz}
\usetikzlibrary{arrows,decorations.markings, cd}
\tikzstyle arrowstyle=[scale=1]
\tikzstyle directed=[postaction={decorate,
decoration={markings,mark=at position .65 with {\arrow[arrowstyle]{stealth}}}}]
\usepackage{mathdots}
\usepackage{quiver}

\newcolumntype{L}{>{$}l<{$}} 
\newcolumntype{C}{>{$}c<{$}}
\newtheorem{theorem}{Theorem}[section]
\newtheorem{lemma}[theorem]{Lemma}
\newtheorem{cor}[theorem]{Corollary}
\newtheorem{prop}[theorem]{Proposition}
\newtheorem{setup}[theorem]{Setup}
\theoremstyle{definition}
\newtheorem{definition}[theorem]{Definition}
\newtheorem{example}[theorem]{Example}

\newtheorem{obs}[theorem]{Observation}
\newtheorem{notation}[theorem]{Notation}

\newtheorem{chunk}[theorem]{}
\theoremstyle{remark}
\newtheorem{remark}[theorem]{Remark}

\newtheorem{the context}[theorem]{The Context}

\numberwithin{equation}{theorem}
\numberwithin{equation}{section}

\newcommand{\cat}[1]{\mathcal{#1}}

\newcommand{\grade}{\operatorname{grade}}

\newcommand{\tor}{\operatorname{Tor}}
\newcommand{\im}{\operatorname{Im}}

\newcommand{\Ker}{\operatorname{Ker}}

\newcommand{\ideal}[1]{\mathfrak{#1}}
\newcommand{\m}{\ideal{m}}

\newcommand{\bbz}{\mathbb{Z}}
\newcommand{\bbn}{\mathbb{N}}

\renewcommand{\geq}{\geqslant}
\renewcommand{\leq}{\leqslant}
\renewcommand{\ker}{\Ker}

\newcommand{\maps}[5]{\xymatrix{#1 \ar[r]^-{#3} & #2 \\
#4 \ar@{|->}[r] & #5 \\}}
\newcommand{\kos}{\textrm{Kos}}

\def\w{\wedge}

\def\im{\operatorname{im}}

\newcommand{\sgn}{\operatorname{sgn}}

\newcommand{\mdeg}{\operatorname{mdeg}}

\begin{document}
\title{On Restricted Powers of Complete Intersections}

\author{Keller VandeBogert }
\date{\today}

\maketitle

\begin{abstract}
    A restricted $d$th power of an ideal $I$ is obtained by restricting the exponent vectors allowed to appear on the ``natural" generating set of $I^d$, for some integer $d$. In this paper, we study homological properties of restricted powers of complete intersections. We construct an explicit minimal free resolution for any restricted power of a complete intersection which generalizes the $L$-complex construction of Buchsbaum and Eisenbud. We use this resolution to compute an explicit basis for the Koszul homology which allows us to deduce that the quotient defined by any restricted $d$th power of a complete intersection is a Golod ring for $d \geq 2$. Finally, using techniques of Miller and Rahmati, we show that the minimal free resolution of the quotient defined by any restricted power of a complete intersection admits the structure of an associative DG-algebra. 
\end{abstract}

\section{Introduction}

Let $I = (a_1 , \dots , a_n)$ be an ideal in some commutative ring $R$ and $w = (w_1 , \dots , w_n) \\ \in \bbz_{\geq 0}^n$ any vector. The $w$-restricted $d$th power of $I$ is the ideal generated by all elements of the form $a_1^{d_1} \cdots a_n^{d_n}$, where $d_i \leq w_i$ for each $1 \leq i \leq n$ and $d_1 + \cdots + d_n = d$. From a combinatorial perspective, if one imagines all exponent vectors with sum $d$ as being encoded by the dilated $(n-1)$-simplex $d \cdot \Delta^{n-1}$, then the exponent vectors for the $w$-restricted $d$th power are obtained by cutting $d \cdot \Delta^{n-1}$ by an appropriate collection of hyperplanes. This perspective is employed in, for instance, \cite{almousa2021}, where the results of Almousa, Fl\o ystad, and Lohne are extended to the case of polarizations of restricted powers of the graded maximal ideal (in a polynomial ring) with a characterization dependent on the associated $w$-restricted dilated simplex.

Gasharov, Hibi, and Peeva have also considered $w$-stable monomial ideals, where the idea is similar: simply restrict the minimal generating set of a stable monomial ideal to all monomials appearing with exponent vector bounded above by $w$. As it turns out, many of the properties possessed by stable ideals are inherited by $w$-stable ideals, such as Golodness and a linear resolution (in the equigenerated case). Moreover, since the Eliahou-Kervaire resolution is a naturally multigraded minimal free resolution, the minimal free resolution of a $w$-stable ideal is obtained by restricting to all multidegrees that are also bounded above by $w$. This recovers and generalizes the \emph{squarefree} Eliahou-Kervaire resolution, introduced in \cite{aramova1998squarefree}.

In this paper, we show that the idea of ``restricting multidegrees" can be extended to arbitrary $w$-restricted powers of complete intersections (which need not be monomial ideals). We introduce a generalization of the well-known $L$-complexes of Buchsbaum and Eisenbud (see \cite{buchsbaum1975generic}) and prove that these complexes yield a minimal free resolution of $w$-restricted powers of complete intersections. We also give an explicit description of the basis elements of the Koszul homology of any restricted power of a complete intersection by lifting basis elements of the aforementioned minimal free resolution to the Koszul homology algebra. This allows us to prove that for $d \geq 2$, the $d$th restricted power of a complete intersection always defines a Golod ring. Finally, we employ recent results of Miller and Rahmati \cite{miller2020transferring} to show that, for large enough characteristic, the minimal free resolution of the quotient defined by the restricted power of any complete intersection admits the structure of an associative DG-algebra.

\begin{figure}
    \centering
{\tiny \[\begin{tikzcd}
	&& {(2,0,0)} &&&&& {(2,0,0)} \\
	& {(1,1,0)} && {(1,0,1)} & { } & { } & {(1,1,0)} && {(1,0,1)} \\
	{(0,2,0)} && {(0,1,1)} && {(0,0,2)} &&& {(0,1,1)}
	\arrow[no head, from=1-3, to=2-2]
	\arrow[no head, from=1-3, to=2-4]
	\arrow[no head, from=2-4, to=3-5]
	\arrow[no head, from=3-3, to=3-5]
	\arrow[no head, from=2-4, to=3-3]
	\arrow[no head, from=2-4, to=2-2]
	\arrow[no head, from=2-2, to=3-1]
	\arrow[no head, from=3-1, to=3-3]
	\arrow[no head, from=2-2, to=3-3]
	\arrow[from=2-5, to=2-6]
	\arrow[no head, from=1-8, to=2-7]
	\arrow[no head, from=2-7, to=3-8]
	\arrow[no head, from=1-8, to=2-9]
	\arrow[no head, from=2-9, to=3-8]
	\arrow[no head, from=2-9, to=2-7]
\end{tikzcd}\]}
    \caption{Restriction of the dilated $1$-simplex $2 \cdot \Delta^1$ to the vector $(2,1,1)$.}
    \label{fig:my_label}
\end{figure}
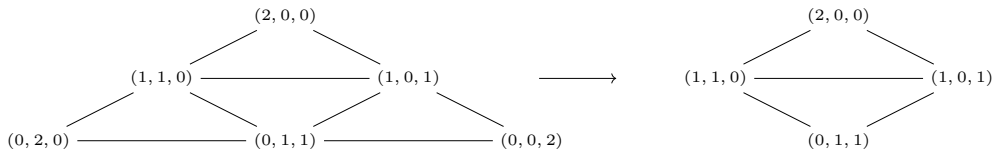

The paper is organized as follows. In Section \ref{sec:background}, we introduce notation and conventions that will be used in the rest of the paper. This includes the definition of a Golod ring in terms of trivial Massey operations and a refresher on much of the content of \cite{miller2020transferring} dealing with transferring algebra structures along special deformation retracts. In Section \ref{sec:theMFR}, we define the pieces that constitute the building blocks of the minimal free resolution of quotients defined by restricted powers of complete intersections. The construction generalizes the original construction of Buchsbaum and Eisenbud, and we give a brief and self-contained proof of acyclicity. 

In Section \ref{sec:koszulAndGolod}, we use the minimal free resolution of Section \ref{sec:theMFR} to obtain information about the Koszul homology algebra and Golodness. More precisely, we are able to find an explicit lift of the basis elements in the minimal free resolution to the Koszul homology algebra. With respect to this basis of the Koszul homology, we find that the trivial Massey operation with $\mu (h_1 , \dots , h_k) = 0$ for $k \geq 2$ is well-defined, allowing us to deduce Golodness of all $w$-restricted $d$th powers of complete intersections, for $d \geq 2$. This also gives an explicit minimal free resolution of the residue field due to a well-known construction of Golod.

In Section \ref{sec:DGA}, we prove that the minimal free resolution constructed in Section \ref{sec:theMFR} admits the structure of an associative DG-algebra. The proof of this fact relies heavily on recent techniques developed by Miller and Rahmati in \cite{miller2020transferring}. In particular, we construct an explicit algebra structure on an associated total complex obtained by restricting multidegrees, and show that the scaled de Rham map satisfies the generalized Leibniz rule with respect to this product. As a consequence, the perturbation lemma will allow us to transfer the product structure on this total complex to the generalized $L$-complexes of Section \ref{sec:theMFR}, immediately extending the application to the Buchsbaum-Eisenbud $L$-complexes given in \cite{miller2020transferring}.

\section{Restricted Powers, Golod Rings, and Transferring Algebra Structures}\label{sec:background}

The purpose of this section is to provide background and conventions for the material to be used for the rest of the paper. The main focal points of this section are the definition of Golod (see Definition \ref{def:Golod}) and the terminology introduced in the latter half, culminating in Proposition \ref{prop:perturbForDG}. For a general overview of DG-algebra techniques, including exposition on Golod rings, the reader is encouraged to consult Avramov \cite{avramov1998infinite}. The material on deformation retracts and the perturbation lemma comes from \cite{miller2020transferring}. We begin this section by defining the $w$-restricted $d$th power of a complete intersection.

\begin{definition}
Let $F$ be a free $R$-module of rank $n$ with basis $f_1 , \dots , f_n$. Let $w = (w_1 , \dots , w_n) \in \bbn^n$ be any vector and let $\psi : F \to R$ be any $R$-module homomorphism. Then the $w$-\emph{restricted $d$th power} of $\im (\psi)$, denoted $(\im \psi)^d_w$, is defined to be the ideal generated by all elements of the form
\begingroup\allowdisplaybreaks
\begin{align*}
    \psi( f_1)^{\alpha_1} \cdot \psi(f_2)^{\alpha_2} \cdots \psi (f_n)^{\alpha_n}, \quad &\textrm{where} \ \alpha_i \leq w_i \ \textrm{for all} \ 1 \leq i \leq n \\
    &\quad \textrm{and} \ \alpha_1 + \cdots + \alpha_n = d.
\end{align*}
\endgroup
\end{definition}

The idea of restricted powers has been considered before by other authors in different contexts. Indeed:
\begin{enumerate}
    \item If $w = (d,d, \dots , d)$, then $(\im \psi)_w^d = (\im \psi)^d$ is just the $d$th power of $\im \psi$.
    \item If $w = (1,1,\dots , 1)$, then $(\im \psi)_w^d$ is the \emph{squarefree} power of $\im \psi$.
    \item If $R = k[x_1 , \dots , x_n]$ and $\im \psi = (x_1 , \dots , x_n)$, then $(\im \psi)^d_w$ is an example of a $w$-stable ideal as considered by Gasharov, Hibi, and Peeva in \cite{gasharov2002resolutions}.
    \item Retaining the notation of $(3)$, let $P(d)$ denote the set of all integer partitions of $d$. Then $(\im \psi)^d_w = I_{P(d)}$, in the notation of the $S_n$-invariant ideals considered by Raicu in \cite{raicu2021regularity}.
    \item In \cite{almousa2021}, it is shown that any restricted power of the graded maximal ideal has a cellular resolution, and a complete characterization for polarizations of all such ideals is given.
\end{enumerate}

\begin{notation}
The notation $(F_\bullet , d_\bullet)$ will denote a complex $F_\bullet$ with differentials $d_\bullet$. When no confusion may occur, $F$ may be written, where the notation $d^F$ is understood to mean the differential of $F$ (in the appropriate homological degree).

Given a complex $F_\bullet$ as above, elements of $F_n$ will often be denoted $f_n$, without specifying that $f_n \in F_n$.
\end{notation}

\begin{definition}\label{def:dga}
A \emph{differential graded algebra} $(F,d)$ (DG-algebra) over a commutative Noetherian ring $R$ is a complex of finitely generated free $R$-modules with differential $d$ and with a unitary, associative multiplication $F \otimes_R F \to F$ satisfying
\begin{enumerate}[(a)]
    \item $F_i F_j \subseteq F_{i+j}$,
    \item $d_{i+j} (x_i x_j) = d_i (x_i) x_j + (-1)^i x_i d_j (x_j)$,
    \item $x_i x_j = (-1)^{ij} x_j x_i$, and
    \item $x_i^2 = 0$ if $i$ is odd,
\end{enumerate}
where $x_k \in F_k$.
\end{definition}

The next definition will be essential for defining Golod rings. If $z_\lambda$ is a cycle in some complex $A$, then the notation $[z_\lambda]$ denotes the homology class of $z_\lambda$. In the following definition, $\overline{a} := (-1)^{|a|+1} a$, where $|a|$ denotes the homological degree of $a \in A$.  

\begin{definition}
Let $A$ be a DG-algebra with $H_0 (A) \cong k$. Then $A$ admits a \emph{trivial Massey operation} if for some $k$-basis $\mathcal{B} = \{ h_\lambda \}_{\lambda \in \Lambda}$, there exists a function
$$\mu : \coprod_{i=1}^\infty \cat{B}^i \to A$$
such that
\begingroup\allowdisplaybreaks
\begin{align*}
    &\mu ( h_\lambda) = z_\lambda \quad \textrm{with} \quad [z_\lambda] = h_\lambda, \ \textrm{and} \\
    &d \mu (h_{\lambda_1} , \dots , h_{\lambda_p} ) = \sum_{j=1}^{p-1} \overline{\mu (h_{\lambda_1} , \dots , h_{\lambda_j})} \mu (h_{\lambda_{j+1}} , \dots , h_{\lambda_p}). 
\end{align*}
\endgroup
\end{definition}

Observe that taking $p=2$ in the above definition yields that $H_{\geq 1} (A)^2 = 0$, so the induced algebra structure on $H(A)$ is totally trivial for a DG-algebra admitting a trivial Massey operation.

\begin{definition}\label{def:Golod}
Let $(R,\m)$ be a local ring and let $K^R$ denote the Koszul complex on the generators of $\m$. If $K^R$ admits a trivial Massey operation $\mu$, then $R$ is called a \emph{Golod ring}. 
\end{definition}

\begin{remark}
It is worth mentioning that Definition \ref{def:Golod} is equivalent to saying that the Poincar\'e series $P^R_k (t)$ of $R$ attains a coefficient-wise inequality originally established by Serre (this is often given as the definition of Golodness, but Definition \ref{def:Golod} will be much more convenient for our purposes). As observed above, the Koszul homology algebra of any Golod ring has trivial multiplication in positive homological degrees; it is worth noting that for rings of projective dimension $\geq 4$, there exist non-Golod rings with trivial Koszul homology algebras (see work by Katth\"an \cite{katthan2017non}).
\end{remark}

Next, we recall some of the terminology and conventions established in \cite{miller2020transferring} for transferring algebra structures. The main goals for the remainder of this section are the establishment of Proposition \ref{prop:perturbForDG}, which furnishes conditions for which an algebra structure can be transferred along a perturbed deformation retract.

\begin{definition}
Let $F_\bullet$ and $G_\bullet$ be two complexes. A deformation retract is a quasi-isomorphism of complexes
\[\begin{tikzcd}
	{F_\bullet} & {G_\bullet}
	\arrow["p", shift left=1, from=1-1, to=1-2]
	\arrow["i", shift left=1, from=1-2, to=1-1]
\end{tikzcd}\]
satisfying:
\begin{enumerate}
    \item $p \circ i = 1$ and
    \item $i \circ p \simeq 1$ via some homotopy $h$ on $F_\bullet$.
\end{enumerate}
A deformation retract is \emph{special} if, furthermore, one has:
\begin{enumerate}
    \item $h\circ i = 0$,
    \item $p \circ h = 0$, and
    \item $h^2 = 0$.
\end{enumerate}
Given a deformation retract
\[\begin{tikzcd}
	{F_\bullet} & {G_\bullet,}
	\arrow["p", shift left=1, from=1-1, to=1-2]
	\arrow["i", shift left=1, from=1-2, to=1-1]
\end{tikzcd}\]
with associated homotopy $h$, a perturbation is a map $\delta$ such that $d^F + \delta$ is a differential on $F_\bullet$; that is, $(d^F + \delta)^2 = 0$. The perturbation $\delta$ is \emph{small} if $1 - \delta h$ is invertible. 
\end{definition}

Note that if $F_\bullet$ is a finite complex, then every perturbation is locally nilpotent and hence small.

\begin{setup}\label{set:perturbationSetup}
Let
\[\begin{tikzcd}
	{F_\bullet} & {G_\bullet,}
	\arrow["p", shift left=1, from=1-1, to=1-2]
	\arrow["i", shift left=1, from=1-2, to=1-1]
\end{tikzcd}\]
be a special deformation retract with associated homotopy $h$. Assume that $\delta$ is a small perturbation on $F_\bullet$ and let $A := (1 - \delta h )^{-1} \delta$. Define the following data:
\begin{enumerate}
    \item $i_\infty := i + h A i$,
    \item $p_\infty := p + p A h$,
    \item $d^F_\infty := d^F + \delta$,
    \item $d^G_\infty := d^G + p A i$, and
    \item $h_\infty := h + h A h$.
\end{enumerate}
\end{setup}

\begin{lemma}[Perturbation Lemma]
Adopt notation and hypotheses as in Setup \ref{set:perturbationSetup}. Then the data
\[\begin{tikzcd}
	{(F_\bullet , d^F_\infty)} & {(G_\bullet, d^G_\infty)}
	\arrow["p_\infty", shift left=1, from=1-1, to=1-2]
	\arrow["i_\infty", shift left=1, from=1-2, to=1-1]
\end{tikzcd}\]
is a special deformation retract with associated homotopy $h_\infty$.
\end{lemma}

\begin{definition}
Let $F_\bullet$ be a complex of $R$-modules equipped with a product, and let $h : F_\bullet \to F_\bullet$ be any graded map. Then $h$ satisfies the \emph{generalized Leibniz rule} if for every $f$, $f' \in F_\bullet$,
$$h(f \cdot f') \subseteq h(f) X + X h (f').$$
\end{definition}

The following Proposition is a slight variant on Proposition $3.5$ of \cite{miller2020transferring}; in the original statement, it was assumed that $F_\bullet$ was a DG-algebra with the unperturbed differential. It turns out that this hypothesis is unnecessary, and that the existence of an associative product for which the associated homotopy satisfies the generalized Leibniz rule is sufficient.

\begin{prop}\label{prop:perturbForDG}
Adopt notation and hypotheses as in Setup \ref{set:perturbationSetup}, and assume furthermore that:
\begin{enumerate}
    \item $(F_\bullet , d^F)$ is an associative algebra (not necessarily satisfying the Leibniz rule),
    \item $d^F_\infty = d^F + \delta$ satisfies the Leibniz rule, and
    \item $h$ satisfies the generalized Leibniz rule.
\end{enumerate}
Then $(G_\bullet , d^G_\infty)$ is an associative DG-algebra with product:
$$g \cdot_G g' := p_\infty (i_\infty (g) \cdot i_\infty (g')).$$
Moreover, the map $i_\infty : (G_\bullet , d^G_\infty) \to (F_\bullet , d^F_\infty)$ is a morphism of DG-algebras. 
\end{prop}

\begin{proof}
The proof is essentially identical to that of \cite[Proposition 1.4]{miller2020transferring}, where one only needs to verify that $(d^F_\infty h_\infty + h_\infty d^F_\infty ) (i_\infty (f )\cdot i_\infty (f')) = 0$; this is a straightforward computation.
\end{proof}

\section{The Minimal Free Resolution}\label{sec:theMFR}

In this section, we construct an explicit minimal free resolution for the quotient defined by the $w$-restricted $d$th power of a complete intersection, for any vector $w$ and power $d$. The construction is highly reminiscent of the original construction of Buchsbaum and Eisenbud; the proof of exactness here is self contained and is distinct from the proof given in \cite{buchsbaum1975generic}. Indeed, the proof of Theorem \ref{thm:theMFR} was inspired by the proof of exactness given in \cite[Theorem 2.12]{el2014artinian}.

Let $F$ be a free $R$-module of rank $n$ with basis elements $f_1 , \dots , f_n$, with $\psi : F \to R$ any map such that $\psi(f_1) , \dots , \psi(f_n)$ forms a regular sequence. Throughout this section, we will assume either:
\begin{enumerate}
    \item $R$ is a Noetherian local ring, or
    \item $R$ is a $\bbz$-graded ring, in which case the elements $\psi( f_1) , \dots , \psi (f_n)$ are assumed to be homogeneous of positive degree.
\end{enumerate}
Both of the above conditions guarantee that any subsequence of the regular sequence $\psi (f_1) , \dots , \psi (f_n)$ remains regular. We begin by establishing the following notation, which will be employed tacitly for the remainder of the paper.

\begin{notation}\label{not:beginningNotation}
Let $F$ be a free $R$-module of rank $n$ with basis elements $f_1 , \dots , f_n$. Given integers $a$ and $b \geq 0$, the following notation will be used for conciseness:
$$\bigwedge^a := \bigwedge^a F, \qquad S_b := S_b(F),$$
denoting the exterior and symmetric algebras, respectively. Let $\sigma = (\sigma_1 < \cdots < \sigma_a)$ be an indexing set and $\alpha = (\alpha_1 , \dots , \alpha_n)$ be an exponent vector with $|\alpha| := \alpha_1 + \cdots + \alpha_n =b$. Then, the following notation will be used:
$$f_\sigma := f_{\sigma_1} \w \cdots \w f_{\sigma_a} \in \bigwedge^a, \qquad f^\alpha := f_1^{\alpha_1} \cdots f_n^{\alpha_n} \in S_b.$$
Given any integer $1 \leq i \leq n$, the notation $\epsilon_i$ will denote the vector with a $1$ in the $i$th spot and $0$'s elsewhere. Throughout the paper, for any two vectors $\alpha$, $\beta \in \bbz^n$, write
$$\alpha \leq \beta \iff \alpha_i \leq \beta_i \ \textrm{for all} \ 1 \leq i \leq n.$$
\end{notation}

\begin{definition}\label{def:mgr}
Adopt notation as in Notation \ref{not:beginningNotation}. Then the \emph{multidegree} of an element $f_\sigma \otimes f^\alpha \in \bigwedge^a \otimes S_b$, denoted $\mdeg$, is defined as
$$\mdeg (f_\sigma \otimes f^\alpha ) := (\epsilon_{\sigma_1} + \dots + \epsilon_{\sigma_a} + \alpha).$$
\end{definition}

The multigrading of Definition \ref{def:mgr} should cause no confusion, since it is induced by the natural choice of multigrading on the symmetric and exterior algebras $S_\bullet$ and $\bigwedge^\bullet$, if each $f_i$ is given multidegree $\epsilon_i$.

\begin{setup}\label{set:thewLcomps}
Adopt notation as in Notation \ref{not:beginningNotation} and let $\psi : F \to R$ be such that $\grade \im (\psi) = n$. Let $w = (w_1 , \dots , w_n)$ be any vector and set
$$(\bigwedge^a \otimes S_b )_w := \{ f_\sigma \otimes f^\alpha \mid \mdeg(f_\sigma \otimes f^\alpha ) \leq w \}.$$
Let $\kappa_{a,b}^w : (\bigwedge^a \otimes S_b )_w \to (\bigwedge^{a-1} \otimes S_{b+1} )_w$ be the map induced by the tautological Koszul differential:
\begingroup\allowdisplaybreaks
\begin{align*}
    \bigwedge^a \otimes S_b &\xrightarrow{\textrm{comult.} \otimes 1} \bigwedge^{a-1} \otimes F \otimes S_b \\
    &\xrightarrow{1 \otimes \textrm{mult.}} \bigwedge^{a-1} \otimes S_{b+1}. \\
\end{align*}
\endgroup
Observe that $\kappa_{a,b}^w$ is well defined because it preserves multidegree. With this notation, define $L^a_{b,w} (F) = L^a_{b,w} := \ker \kappa_{a,b}^w \subseteq (\bigwedge^a \otimes S_b)_w$. 
\end{setup}

As it turns out, the basis elements for these modules are actually quite comprehensible. The proof of the following proposition is a consequence of the much more general statement given in \cite[Proposition 2.1.4]{weyman2003} combined with the fact that all tableaux appearing in a given straightening relation have the same multidegree.

\begin{prop}
For every $a,b \geq 0$, the module $L^a_{b,w}$ has basis represented by the set of semistandard tableaux with multidegree bounded above by $w$.
\end{prop}

\begin{example}
Let $R = k[x_1 , \dots , x_3]$ and $F = Rf_1 \oplus Rf_2 \oplus Rf_3$ with $\psi : F \to R$ induced by sending $f_i \mapsto x_i$. Suppose that $w = (3,1,1)$; then
$$(x_1 , x_2 , x_3)^3_w = (x_1^3 , x_1^2 x_2 , x_1^2 x_3 , x_1x_2 x_3).$$
Likewise, the module $L^1_{3 , w} (F)$ has basis represented by the tableaux
$$ \ytableausetup
{boxsize=1.7em}
\begin{ytableau}
1 & 1 & 1\\
2 \\
\end{ytableau}, \quad \ytableausetup
{boxsize=1.7em}
\begin{ytableau}
1 & 1 & 1\\
3 \\
\end{ytableau}, \quad 
\ytableausetup
{boxsize=1.7em}
\begin{ytableau}
1 & 1 & 2\\
3 \\
\end{ytableau}, \quad 
\ytableausetup
{boxsize=1.7em}
\begin{ytableau}
1 & 1 & 3 \\
2 \\
\end{ytableau}.$$
\end{example}

The following observation is a straightforward verification which shows that the horizontal and vertical differentials of Figure \ref{fig:doubleComp} anticommute.

\begin{obs}
Adopt notation and hypotheses as in Setup \ref{set:thewLcomps}. Let $\kos : \bigwedge^a \to \bigwedge^{a-1}$ denote the Koszul differential induced by $\psi$, for any given $a \geq 0$. Then $\kos \otimes 1 : \bigwedge^\bullet \otimes S_\bullet \to \bigwedge^\bullet \otimes S_\bullet$ and $\kappa_{\bullet, \bullet}$ anticommute; that is,
$$(\kos \otimes 1) \circ \kappa_{a,b}^w  = - \kappa_{a-1,b}^w \circ (\kos \otimes 1).$$
\end{obs}

\begin{notation}\label{not:convenienceNotation}
For conciseness and ease of notation, subscripts indicating homological degrees will often be omitted. Moreover, the vector $w$ as in Setup \ref{set:thewLcomps} will often be omitted in the notation $\kappa_{a,b}^w$, and the much simpler notation $\kappa$ may be used.

Likewise, for any map $\psi : F \to R$, there is a naturally induced map $S_d (\psi ) : S_d \to R$ defined by sending $f^\alpha \mapsto \psi(f)^{\alpha_1} \cdots \psi(f_n)^{\alpha_n}$. This map will also be denoted simply by $\psi$.

Given a double complex $E_{\bullet , \bullet}$ with anticommuting vertical and horizontal differentials $d^{E,v}$ and $d^{E,h}$, respectively, the total complex $T(E_{\bullet , \bullet})$ will have differential $d^{E,h} - d^{E,v}$.  
\end{notation}

The following proposition will be helpful in the proof of Theorem \ref{thm:theMFR}. Recall that the assumptions put forth at the beginning of this section imply that any subsequence of our regular sequence is again a regular sequence.

\begin{prop}\label{prop:exactnessOfE}
Adopt notation and hypotheses as in Setup \ref{set:thewLcomps}. Let $E$ denote the double complex of Figure \ref{fig:doubleComp}. Then,
\begin{enumerate}
    \item all rows of $E$ except for the bottom-most row are exact, and
    \item the complexes $(\bigwedge^\bullet \otimes S_b)_w$ for $0 \leq b \leq d-1$ are exact in homological degrees $\geq 1$.
\end{enumerate}
\end{prop}

\begin{proof}
Since the rows of $E$ are obtained from restricting the multidegrees appearing in the truncated tautological Koszul complex, they must be exact, except for the bottom-most row. At the bottom, the homology is isomorphic to $R$.

To see that the complexes $(\bigwedge^\bullet \otimes S_b)_w$ are acyclic, observe that if $f_{s} \otimes f^\alpha \in (\bigwedge^1 \otimes S_b)_w$ for all $s \in S$, where $S$ is some subset of $[n] = \{ 1 , \dots , n \}$, then $f_\sigma \otimes f^\alpha \in (\bigwedge^{|\sigma|} \otimes S_b)_w$ for every $\sigma \subset S$. Thus, $(\bigwedge^\bullet \otimes S_b)_w$ decomposes as a direct sum of Koszul complexes on subsets of the set $\{ \psi (f_1) , \dots , \psi (f_n) \}$. 
\end{proof}

\begin{definition}\label{def:theLwComp}
Adopt notation and hypotheses as in Setup \ref{set:thewLcomps}. Then $L^w (\psi , d)$ denotes the complex
$$0 \to L_{d,w}^{n-1} (F) \xrightarrow{\kos \otimes 1} \cdots \xrightarrow{\kos \otimes 1} L_{d,w}^{1} (F) \xrightarrow{\kos \otimes 1} (S_d)_w \xrightarrow{\psi} R \to 0.$$
\end{definition}

We finally arrive at the main result of this section:

\begin{theorem}\label{thm:theMFR}
Adopt notation and hypotheses as in Setup \ref{set:thewLcomps}. The complex $L^w (\psi , d)$ of Definition \ref{def:theLwComp} is the minimal free resolution of $R / (\im \psi)^d_w$.
\end{theorem}

\begin{figure}
    \centering
\[\begin{tikzcd}
	&&&& 0 \\
	&& 0 & {(\bigwedge^n \otimes S_{d-1})_w} & {L^{n-1}_{d,w}} \\
	& \iddots & \vdots & \vdots & \vdots \\
	\vdots && \iddots & {(\bigwedge^2 \otimes S_{d-1})_w} & {L_{d,w}^1} \\
	&& \iddots & {(\bigwedge^{1} \otimes S_{d-1})_w} & {L^0_{d,w} } \\
	\vdots & \vdots & \vdots & \iddots \\
	{(\bigwedge^2 \otimes S_0)_w} & {(\bigwedge^1 \otimes S_1)_w} & {(\bigwedge^0 \otimes S_2)_w} \\
	{(\bigwedge^1 \otimes S_0)_w} & {(\bigwedge^0 \otimes S_1)_w} \\
	{( \bigwedge^0 \otimes S_0 )_w}
	\arrow["\kappa"', from=4-3, to=4-4]
	\arrow["\kappa"', from=5-3, to=5-4]
	\arrow["{\kos \otimes 1}", from=3-4, to=4-4]
	\arrow["{\kos \otimes 1}", from=3-5, to=4-5]
	\arrow["{\kos \otimes 1}", from=4-5, to=5-5]
	\arrow["\kappa", from=7-2, to=7-3]
	\arrow["\kappa", from=7-1, to=7-2]
	\arrow[from=8-1, to=8-2]
	\arrow["{\kos \otimes 1}", from=7-1, to=8-1]
	\arrow["{\kos \otimes 1}", from=7-2, to=8-2]
	\arrow["{\kos \otimes 1}", from=8-1, to=9-1]
	\arrow["{\kos \otimes 1}", from=4-4, to=5-4]
	\arrow["\kappa"', from=5-4, to=5-5]
	\arrow["\kappa"', from=4-4, to=4-5]
	\arrow["\kappa"', from=2-4, to=2-5]
	\arrow[from=2-3, to=2-4]
	\arrow[from=2-3, to=3-3]
	\arrow["{\kos \otimes 1}", from=2-4, to=3-4]
	\arrow["{\kos \otimes 1}", from=2-5, to=3-5]
	\arrow["{\kos \otimes 1}", from=6-1, to=7-1]
	\arrow["{\kos \otimes 1}", from=6-2, to=7-2]
	\arrow["{\kos \otimes 1}", from=6-3, to=7-3]
	\arrow[from=1-5, to=2-5]
\end{tikzcd}\]
    \caption{The double complex used for the proof of Theorem \ref{thm:theMFR}; recall the conventions established in Notation \ref{not:convenienceNotation}.}
    \label{fig:doubleComp}
\end{figure}

\begin{proof}
Throughout the proof, let $T(-)$ denote the totalization of a double complex. Let $E$ denote the double complex of Figure \ref{fig:doubleComp}, $E'$ the double complex obtained by deleting the bottom most $\bigwedge^0 \otimes S_0$ term, and $E''$ the rightmost nontrivial column, which is $L^w (\psi, d)$ without $R$ in homological degree $0$. Index the total complex $T(E)$ such that $T(E)_0 = \Big( \bigoplus_{i=0}^{d-1} (\bigwedge^0 \otimes S_i)_w \Big) \oplus L^0_{d,w}$. Then the short exact sequence of total complexes
$$0 \to T(E') \to T(E) \to T(E/E') \to 0$$
combined with Proposition \ref{prop:exactnessOfE} (1) yields that 
$$H_i (T(E)) = H_i (T(E')) = \begin{cases} R & \textrm{if} \  i=0 \\
0 & \textrm{otherwise}.
\end{cases}$$
Likewise, the short exact sequence
$$0 \to T(E'') \to T(E) \to T(E/E'') \to 0$$
combined with Proposition \ref{prop:exactnessOfE} (2) yields that $H_i (T(E'')) = 0$ for $i \geq 1$ and there is an inclusion of homology
$$H_0 (T(E'')) \hookrightarrow H_0 (T(E)).$$
By definition, this inclusion lifts a cycle $f^\alpha = f_{i_1} \cdots f_{i_d} \in (S_d)_w = L^0_{d,w}$ to a cycle of $T(E)_0$; one such lift is given by
$$f^\alpha + \sum_{j=1}^d \psi( f_{i_1} \cdots f_{i_j}) f_{i_{j+1}} \cdots f_{i_d}.$$
Composing with the isomorphism $H_0 (T(E)) \xrightarrow{\sim} R$ induces the map $f^\alpha \mapsto \psi (f^\alpha)$, whence augmenting $E''$ by the map $(S_d)_w \xrightarrow{\psi} R$ remains acyclic. 
\end{proof}

\section{Koszul Homology and Golodness}\label{sec:koszulAndGolod}

In this section, we study the Koszul homology of restricted powers of complete intersections in regular local rings. Since the differentials in the complexes of Definition \ref{def:theLwComp} are induced by Koszul differentials, one can explicitly compute the Koszul lift to the homology algebra. The main result of this section is Corollary \ref{cor:Golodness}, but most of the work is done in the proof of Proposition \ref{prop:tensorCycles}. After computing a basis for the Koszul homology, we are able to find a straightforward trivial Massey operation allowing for an easy description of the minimal free resolution of the residue field over the quotient defined by a restricted power of a complete intersection.

Throughout this section, we will assume either:
\begin{enumerate}
    \item $R$ is a local ring, or
    \item $R$ is a standard graded polynomial ring over a field $k$, in which case $\psi (f_1 ) , \dots , \psi (f_n)$ are assumed to be homogeneous of positive degree.
\end{enumerate}

\begin{definition}
Let $M$ be a finitely generated $R$-module. The \emph{Koszul homology} of $M$, denoted $H_\bullet (M)$, is defined to be the homology of $M \otimes_R K_\bullet$; that is
$$H_\bullet (M) := H_\bullet (M \otimes_R K_\bullet),$$
where $K_\bullet$ denote the Koszul complex resolving the residue field, $k$.
\end{definition}

Let $A_\bullet$ and $B_\bullet$ be complexes with $H_0 (A) = R$ and $H_0 (B) = S$. Recall that there is a functorial isomorphism
$$H_\bullet (A \otimes S) \cong H_\bullet (R \otimes B)$$
induced by the natural projections
\[\begin{tikzcd}
	A & {A \otimes B} & B
	\arrow[from=1-2, to=1-1]
	\arrow[from=1-2, to=1-3]
\end{tikzcd}\]
The isomorphism can be explicitly described as follows:
\begin{enumerate}
    \item Choose a cycle $z_1$ in $A \otimes S$ representing a basis element of $H_\bullet (A \otimes S)$.
    \item Lift $z_1$ to a cycle $z_2$ in $A \otimes B$.
    \item Project $z_2$ onto a cycle $z_3$ in $R \otimes B$, then descend to homology.
\end{enumerate}

\begin{definition}\label{def:thePhiMaps}
Adopt notation and hypotheses as in Setup \ref{set:thewLcomps}. Let $K_\bullet$ denote the Koszul complex resolving the residue field $k$. The map $\phi^i_j : \bigwedge^i F \to K_j \otimes \bigwedge^{i-j} F$ is defined to be the composition
\begingroup\allowdisplaybreaks
\begin{align*}
    \bigwedge^i F &\xrightarrow{\textrm{comult}} \bigwedge^j F \otimes \bigwedge^{i-j} F \\
    &\xrightarrow{\textrm{lift} \otimes 1} K_j \otimes \bigwedge^{i-j} F ,\\
\end{align*}
\endgroup
where $\textrm{lift} : \bigwedge^j \to K_j$ denotes the projection onto $K_j$ of the cycle in $(\bigwedge^\bullet F \otimes K_\bullet)_j$ that represents any $f_\tau \in \bigwedge^j$.
\end{definition}

In the notation of Definition \ref{def:thePhiMaps}, let $z_\tau$ denote the lift of any $f_\tau \in \bigwedge^j$ to the Koszul algebra. Then the map $\phi_j^i$ of Definition \ref{def:thePhiMaps} can be written explicitly:
\begingroup\allowdisplaybreaks
\begin{align*}
    \phi_j^i (f_\sigma) = \sum_{\substack{\tau \subseteq \sigma \\
    |\tau| = j \\}} \sgn(\tau 
\subset \sigma) z_\tau \otimes f_{\sigma \backslash \tau}, 
\end{align*}
\endgroup
where $\sgn$ is the sign of the permutation that reorders $(\sigma \backslash \tau ) \cup \tau$ into ascending order. The next proposition tells us that the $\phi_i^j$ maps can be used to give an explicit lift of basis elements of the complex of Definition \ref{def:theLwComp} to the tensor product complex $L^w (\psi , d) \otimes K_\bullet$.

\begin{prop}\label{prop:tensorCycles}
Adopt notation and hypotheses as in Setup \ref{set:thewLcomps}. Let $\bigwedge^\bullet F$ denote the Koszul complex induced by the map $\psi : F \to R$ and let $(\bigwedge^\bullet F \otimes S_d)_w$ be the complex obtained by restricting to terms with multidegrees bounded above by $w$. Given any $f_\sigma \otimes f^\alpha \in (\bigwedge^i F \otimes S_{d-1})_w$, the element
$$f_\sigma \otimes f^\alpha + \sum_{j=1}^{i-1} \phi_j^i (f_\sigma) \otimes f^\alpha + \psi (f^\alpha) \phi_i^i (f_\sigma)$$
is a cycle in $\big( (\bigwedge^\bullet F \otimes S_{d-1})_w \otimes K_\bullet \big)_i$. 
\end{prop}

\begin{proof}
Assume $1 \leq j \leq i-1$ (noting that $\phi_0^i$ is the identity). Then:
\begingroup\allowdisplaybreaks
\begin{align*}
    (\kos \otimes 1 )(\phi_j^i (f_\sigma) \otimes f^\alpha) &= (\kos \otimes 1) \Big( \sum_{\substack{\tau \subset \in \sigma \\
    |\tau| = j\\}} \sgn (\tau) z_{\tau} \otimes f_{\sigma \backslash \tau} \otimes f^\alpha \Big) \\
    &= \sum_{\substack{\tau \subset \sigma \\
    |\tau|=j \\}} \sum_{r \in \tau} \sgn(r \in \tau) \sgn(\tau \subset \sigma) \psi(f_r) z_{\tau \backslash r} \otimes f_{\sigma \backslash \tau} \otimes f^\alpha \\
    &= \sum_{\substack{\tau' \subset \sigma \\
    |\tau'| = j-1\\}} \sum_{r \in \sigma \backslash \tau'} \sgn (r \in \sigma \backslash \tau') \sgn (\tau' \subset \sigma) z_{\tau'} \otimes \psi(f_r) f_{(\sigma \backslash \tau') \backslash r} \\
    &= (1 \otimes \kos) (\phi_{j-1}^i (f_\sigma) \otimes f^\alpha) . \\
\end{align*}
\endgroup
In order to see that the above signs are equal, for any $r \in \sigma$ and $\tau' \in \sigma$, let $\tau$ be the set $\tau' \cup r$ ordered in ascending order. One can reorder the set $(\sigma \backslash (\tau' \cup r)) \cup \tau' \cup r$ into ascending order by either:
\begin{enumerate}
    \item First reorder $\tau' \cup r$ into ascending order, then reorder $(\sigma \backslash (\tau' \cup r)) \cup \tau$ into ascending order; this permutation has sign $\sgn ( r \in \tau) \sgn (\tau \subset \sigma)$.
    \item First reorder $(\sigma \backslash (\tau' \cup r)) \cup r$ into ascending order, then reorder $(\sigma \backslash \tau') \cup \tau'$ into ascending order; this permutation has sign $\sgn (r \in \sigma \backslash \tau') \sgn (\tau' \subset \sigma)$.
\end{enumerate}
Since both of the above cases yield a permutation of the same parity, the signs are indeed equal.

In the case $j=i$, one has:
\begingroup\allowdisplaybreaks
\begin{align*}
    (1 \otimes \psi) ( \phi_{i-1}^i (f_\sigma) \otimes f^\alpha ) &= \sum_{r \in \sigma} \sgn(r \in \sigma) \psi (f^\alpha \cdot f_r) z_{\sigma \backslash r} \\
    &=(\kos \otimes 1) (\psi(f^\alpha) z_\sigma). \\
\end{align*}
\endgroup
This completes the proof.
\end{proof}

\begin{cor}\label{cor:Golodness}
Adopt notation and hypotheses as in Setup \ref{set:thewLcomps}. The correspondence $f_\sigma \otimes f^\alpha  \mapsto \psi (f^\alpha) \phi_i^i (f_\sigma)$ induces an isomorphism of homology $L^w (\psi , d) \otimes k \to H_\bullet (R/(\im \psi)^d_w)$. Moreover, the ring $R/ (\im \psi)^d_w$ is Golod whenever $d \geq 2$. 
\end{cor}

\begin{proof}
The first part of the statement is clear by construction of the isomorphism $\tor_\bullet^R ( - , k) \cong H_\bullet (- \otimes K_\bullet )$ combined with Proposition \ref{prop:tensorCycles}. The Golodness follows from noticing that the product of elements of the form $\psi (f^ \alpha) \phi_i^i (f_\sigma)$ in $R/ (\im \psi)^d_w \otimes K_\bullet$ are trivial. This implies that simply choosing $\mu ( h_1 , \dots , h_n) = 0$ for $i >1$ is a well defined trivial Massey operation on the Koszul homology. 
\end{proof}

For any $\ell \geq 1$, let $V_\ell$ denote the free $R$-module with formal basis elements
$$\{ v_{\sigma , \alpha} \mid \ell = |\sigma| + 1 \},$$
where we are thinking of each $v_{\sigma , \alpha}$ as being a formal stand-in for the basis element $\psi (f^\alpha) z_\sigma$ in the Koszul homology algebra. Then, combining Corollary \ref{cor:Golodness} with Golod's construction of the minimal free resolution of the residue field (see, for instance, \cite[Theorem 5.2.2]{avramov1998infinite}), we have:

\begin{cor}
Adopt notation and hypotheses as in Setup \ref{set:thewLcomps}. Let $(T_\bullet, \partial_\bullet)$ denote the complex with
\begingroup\allowdisplaybreaks
\begin{align*}
    T_n &:= \bigoplus_{p+h+i_1 + \cdots + i_p = n} K_h \otimes_R V_{i_1} \otimes_R \cdots \otimes_R V_{i_p}, \\
    \partial_n &: T_n \to T_{n-1}, \\
    \partial_n (a \otimes v_{\sigma^1 \alpha^1} \otimes \cdots \otimes v_{\sigma^p, \alpha^p} ) &= d(a) \otimes v_{\sigma^1 \alpha^1} \otimes \cdots \otimes v_{\sigma^p, \alpha^p} \\
    + (-1)^{|a|} & a \psi(f^{\alpha^1}) f_{\sigma^1} \otimes v_{\alpha^2 , \sigma^2} \otimes \cdots \otimes v_{\alpha^p , \sigma^p}. \\
\end{align*}
\endgroup
Then $T_\bullet$ is the minimal free resolution of the residue field $k$ over $R/(\im \psi)^d_w$.
\end{cor}

\begin{remark}
In a slightly different direction, it is easy to see that $w$-restricted powers of arbitrary monomial ideals are $d$-Golod (hence Golod), in the terminology of \cite{herzog2018koszul}.
\end{remark}

\section{Algebra Structure on the Minimal Free Resolution}\label{sec:DGA}

In this section, we prove that the generalized $L$-complexes of Definition \ref{def:theLwComp} admit the structure of an associative DG-algebra. The methods employed here come from techniques developed by Miller and Rahmati in \cite{miller2020transferring}. The process of constructing the algebra structure consists of a few steps. First, we construct an algebra structure on an associated total complex that surjects onto the complexes of Definition \ref{def:theLwComp}. Next, we observe that the scaled de Rham map (see Lemma \ref{lem:scaleDerham}) satisfies the generalized Leibniz rule with respect to this algebra structure, inducing a special deformation retract. Finally, the result will follow after combining the previous two sentences with Proposition \ref{prop:perturbForDG}.

Let $F$ be a free $R$-module of rank $n$ with basis elements $f_1 , \dots , f_n$, with $\psi : F \to R$ any map such that $\psi(f_1) , \dots , \psi(f_n)$ forms a regular sequence. Throughout this section, we will assume either:
\begin{enumerate}
    \item $R$ is a Noetherian local ring and $n+d$ is a unit, or
    \item $R$ is a $\bbz$-graded ring and $n+d$ is a unit, in which case the elements $\psi( f_1) , \dots , \psi (f_n)$ are assumed to be homogeneous of positive degree.
\end{enumerate}
The characteristic assumptions imposed in $(1)$ and $(2)$ are needed because of the definition of the scaled de Rham differential given in Lemma \ref{lem:scaleDerham}. In the following definition, we show how to build a well-defined algebra structure after restricting multidegrees. 

\begin{definition}\label{def:theProduct}
Let $(S_\bullet)_w$ denote the $R$-submodule of $S_\bullet$ generated by all monomials with multidegree bounded by $w$. Observe that $(S_\bullet)_w$ is not a subalgebra with respect to the ordinary multiplication on $S_\bullet$. However, $(S_\bullet)_w$ may be given an associative algebra structure as follows (notice: this is not necessarily graded):
$$f^\alpha \cdot f^\beta = \psi( f^{\alpha + \beta - \min (\alpha+ \beta , w)}) f^{\min ( \alpha + \beta , w)}.$$
Let $X_d^w$ denote the total complex of the double complex obtained by deleting the rightmost nontrivial column of the double complex in Figure \ref{fig:doubleComp}. Then $X_d^w$ may be given the structure of an associative algebra with product defined as follows:
$$(f_\sigma \otimes f^\alpha) (f_\tau \otimes f^\beta) = \begin{cases} 
0 & \textrm{if} \ \alpha_i + \beta_i \geq w_i, \ \textrm{and} \  i \in \sigma \cup \tau \ \textrm{for some} \ i, \\
f_\sigma \w f_\tau \otimes f^\alpha \cdot f^\beta & \textrm{otherwise}, \\
\end{cases}$$
where in the above, it is understood that if $f^\alpha \cdot f^\beta \in (S_{\geq d})_w$, then the product is $0$.
\end{definition}

\begin{example}
The product of Definition \ref{def:theProduct} may seem unnecessarily complicated at first sight, but it is important to notice that the ``obvious" choice may not be well-defined. For instance, let $R = k[x_1 , x_2]$ and $F = Rf_1 \oplus Rf_2$ with $\psi : F \to R$ induced by sending $f_i \mapsto x_i$. If $w = (1,1)$, then the standard product on $S_\bullet$ does not restrict to a product on $(S_\bullet)_w$ since $f_1 \cdot f_1 = f_1^2 \notin (S_\bullet)_w$. The product of Definition \ref{def:theProduct} pulls the ``overflow" out as a coefficient, so that $f_1 \cdot f_1 = x_1 f_1 \in (S_\bullet)_w$.
\end{example}

\begin{prop}\label{prop:totalCxDG}
Adopt notation and hypotheses as in Setup \ref{set:thewLcomps}. With product as in Definition \ref{def:theProduct}, the total complex $X_d^w$ is an associative DG-algebra.
\end{prop}

\begin{proof}
Let $S := \{ i \mid \alpha_i + \beta_i \geq w_i \ \textrm{and} \ i \in \sigma \cup \tau \}$. If $|S| = 0$, then the proof of the Leibniz rule is essentially identical to that of the Koszul complex.

If $|S| = 1$, then let $\ell$ be the unique integer with $\alpha_\ell + \beta_\ell \geq w_\ell$ and $\ell \in \sigma \cup \tau$. Assume without loss of generality that $\ell \in \sigma$. By definition, $(f_\sigma \otimes f^\alpha) \cdot (f_\tau \otimes f^\beta ) = 0$. On the other hand, one computes:
\begingroup\allowdisplaybreaks
\begin{align*}
    &d(f_\sigma \otimes f^\alpha) (f_\tau \otimes f^\beta) + (-1)^{|\sigma|} (f_\sigma \otimes f^\alpha) d(f_\tau \otimes f^\beta) \\
    =& \sgn (\ell \in \sigma) \Big( -\psi(f_\ell) f_{\sigma \backslash \ell} \w f_\tau \otimes f^\alpha \cdot f^\beta + f_{\sigma \backslash \ell} \w f_\tau \otimes f^{\alpha + \epsilon_\ell} \cdot f^\beta \Big) = 0 . \\
\end{align*}
\endgroup
Finally, if $|S| >1$, then all terms appearing in the Leibniz rule still have trivial multiplication. To conclude the proof, observe that associativity follows by the associativity of the exterior algebra and the product defined in Definition \ref{def:theProduct}. 
\end{proof}

\begin{lemma}\label{lem:scaleDerham}
Adopt notation and hypotheses as in Setup \ref{set:thewLcomps}. Let $h$ denote the scaled de Rham map
$$h_{a,b} := \frac{1}{a+b} \sum_{j=1}^n f_j \otimes \frac{\partial}{\partial f_j} : \bigwedge^a \otimes S_b \to \bigwedge^{a+1} \otimes S_{b-1}.$$
Then,
\begin{enumerate}
    \item $h^2 = 0$,
    \item $\kappa h + h \kappa = 1$, and
    \item $h$ restricts to a contracting homotopy $(\bigwedge^a \otimes S_b)_w \to (\bigwedge^{a+1} \otimes S_{b-1})_w$. 
\end{enumerate}
\end{lemma}

\begin{proof}
The proofs of $(1)$ and $(2)$ may be found in \cite[Lemma 4.8]{miller2020transferring}, and $(3)$ follows because $h$ preserves multidegree.
\end{proof}

\begin{prop}\label{prop:genLeibniz}
Adopt notation and hypotheses as in Setup \ref{set:thewLcomps} and let $h$ denote the scaled de Rham map of Lemma \ref{lem:scaleDerham}. Then $h$ satisfies the generalized Leibniz rule with respect to the product of Definition \ref{def:theProduct}.
\end{prop}

\begin{proof}
Let $f_\sigma \otimes f^\alpha \in \bigwedge^r \otimes S_a$ and $f_\tau \otimes f^\beta \in \bigwedge^s \otimes S_b$. If $(f_\sigma \otimes f^\alpha ) (f_\tau \otimes f^\beta ) = 0$, then there is nothing to prove. Assume that the product is not zero and define $T := \{ i \mid \alpha_i + \beta_i > w_i \}$. For each $i \in T$, let $\alpha_i' < \alpha_i$ and $\beta_i' < \beta_i$ be such that $\alpha_i' + \beta_i = w_i$ and $\alpha_i + \beta_i' = w_i$. Observe that for any $i \in T$,
\begingroup\allowdisplaybreaks
\begin{align*}
    (r+a)h( f_\sigma \otimes f^\alpha) (  \psi (f_i^{\beta_i - \beta_i'} ) f_\tau \otimes f^{\beta - (\beta_i - \beta_i') \epsilon_i} ) &= \sum_{j \notin T} f_j \w f_\sigma \w f_\tau \otimes \frac{\partial (f^\alpha \cdot f^\beta) }{\partial f_j} \\
    &+ f_i \w f_\sigma \w f_\tau \otimes \frac{\partial (f^\alpha \cdot f^\beta) }{\partial f_i}, \quad \textrm{and} \\
    (s+b)(\psi (f_i^{\alpha_i - \alpha_i'} ) f_\sigma \otimes f^{\alpha - (\alpha_i - \alpha_i') \epsilon_i} )h ( f_\tau \otimes f^\beta) &= \sum_{j \notin T} f_\sigma \w f_j \w f_\tau \otimes \frac{\partial (f^\alpha \cdot f^\beta) }{\partial f_j} \\
    &+ f_\sigma \w f_i \w f_\tau \otimes \frac{\partial (f^\alpha \cdot f^\beta) }{\partial f_i}. \\
\end{align*}
\endgroup
Using this, one computes:
\begingroup\allowdisplaybreaks
\begin{align*}
    &(r+a) h( f_\sigma \otimes f^\alpha) \big( (1 - |T|) f_\tau \otimes f^\beta + \sum_{i \in T} \psi (f_i^{\beta_i - \beta_i'} ) f_\tau \otimes f^{\beta - (\beta_i - \beta_i') \epsilon_i} \big) \\
    &+ (-1)^r(s+b)  \big( (1 - |T|) f_\sigma \otimes f^\alpha + \sum_{i \in T} \psi (f_i^{\alpha_i - \alpha_i'} ) f_\sigma \otimes f^{\alpha - (\alpha_i - \alpha_i') \epsilon_i} \big) h ( f_\tau \otimes f^\beta) \\
    =&  \sum_{i=1}^n f_i \w f_\sigma \w f_\tau \otimes \frac{\partial (f^\alpha \cdot f^\beta)}{\partial f_i} \\
    =& (r+s+a+b)h \big( (f_\sigma \otimes f^\alpha ) (f_\tau \otimes f^\beta ) \big) . \\
\end{align*}
\endgroup
Dividing the above by $r+s+a+b$, it follows that $h$ satisfies the generalized Leibniz rule. 
\end{proof}

\begin{chunk}\label{chunk:discussionOfTransfer}
Let us recall the method of transferring algebra structures used in \cite{miller2020transferring} as applied to our situation. The goal of this method is to find a special deformation retract 
$$(X_d^w , \kappa - \kos \otimes 1 ) \rightleftarrows (L^w (\psi , d) , \kos \otimes 1 )$$
that is a perturbation of a special deformation retract, with the homotopy $h$ satisfying the generalized Leibniz rule. Let $\varepsilon$ denote any choice of isomorphism $\bigwedge^0 \otimes S_0 \cong R$. In our situation,
\begin{figure}[H]
    \centering
\[\begin{tikzcd}
	&&&& 0 \\
	&& 0 & {(\bigwedge^n \otimes S_{d-1})_w} & {L^{n-1}_{d,w}} \\
	& \iddots & \vdots & \vdots & \vdots \\
	\vdots && \cdots & {(\bigwedge^2 \otimes S_{d-1})_w} & {L_{d,w}^1} \\
	&& \cdots & {(\bigwedge^{1} \otimes S_{d-1})_w} & {L^0_{d,w} } \\
	\vdots & \vdots & \vdots & {(\bigwedge^0 \otimes S_{d-1})_w} & R \\
	{(\bigwedge^2 \otimes S_0)_w} & {(\bigwedge^1 \otimes S_1)_w} & {(\bigwedge^0 \otimes S_2)_w} \\
	{(\bigwedge^1 \otimes S_0)_w} & {(\bigwedge^0 \otimes S_1)_w} \\
	{( \bigwedge^0 \otimes S_0 )_w} && {}
	\arrow["\kappa", from=7-2, to=7-3]
	\arrow["\kappa", from=7-1, to=7-2]
	\arrow["\kappa", from=8-1, to=8-2]
	\arrow["p", curve={height=12pt}, from=5-4, to=5-5]
	\arrow["p", curve={height=12pt}, from=4-4, to=4-5]
	\arrow["p", curve={height=12pt}, from=2-4, to=2-5]
	\arrow[from=2-3, to=2-4]
	\arrow["p"', curve={height=30pt}, from=9-1, to=6-5]
	\arrow["p", curve={height=12pt}, from=6-4, to=6-5]
	\arrow[from=6-3, to=6-4]
	\arrow[from=5-3, to=5-4]
	\arrow["i"', curve={height=-18pt}, from=6-5, to=9-1]
	\arrow["i", curve={height=12pt}, from=6-5, to=6-4]
	\arrow["i", curve={height=12pt}, from=5-5, to=5-4]
	\arrow["i", curve={height=12pt}, from=4-5, to=4-4]
	\arrow["i", curve={height=12pt}, from=2-5, to=2-4]
	\arrow[from=4-3, to=4-4]
\end{tikzcd}\]
\caption{The unperturbed double complex.}\label{fig:unperturbed}
\end{figure}
\begin{enumerate}
    \item The special deformation retract to be perturbed comes from the double complex of Figure \ref{fig:unperturbed}, with contracting homotopy $h$ coming from Lemma \ref{lem:scaleDerham} and
    \begingroup\allowdisplaybreaks
    \begin{align*}
        i &= \begin{cases}
    h & \textrm{on} \ L^i_{a,w} \\
    \varepsilon^{-1} & \textrm{on} \ R, \\
    \end{cases} \\
    p &= \begin{cases}
    \kappa & \textrm{on} \ (\bigwedge^i \otimes S_{a-1})_w, \ i>0, \\
    \varepsilon & \textrm{on} \ R, \\
    0 & \textrm{otherwise}. \\
    \end{cases}
    \end{align*}
    \endgroup
 \begin{figure}[H]
    \centering
\[\begin{tikzcd}
	&&&& 0 \\
	&& 0 & {(\bigwedge^n \otimes S_{d-1})_w} & {L^{n-1}_{d,w}} \\
	& \iddots & \vdots & \vdots & \vdots \\
	\vdots &&& {(\bigwedge^2 \otimes S_{d-1})_w} & {L_{d,w}^1} \\
	&& \cdots & {(\bigwedge^{1} \otimes S_{d-1})_w} & {L^0_{d,w} } \\
	\vdots & \vdots & \vdots & {(\bigwedge^0 \otimes S_{d-1})_w} & R \\
	{(\bigwedge^2 \otimes S_0)_w} & {(\bigwedge^1 \otimes S_1)_w} & {(\bigwedge^0 \otimes S_2)_w} \\
	{(\bigwedge^1 \otimes S_0)_w} & {(\bigwedge^0 \otimes S_1)_w} \\
	{( \bigwedge^0 \otimes S_0 )_w}
	\arrow["{\kos \otimes 1}", from=3-4, to=4-4]
	\arrow["{\kos \otimes 1}", from=3-5, to=4-5]
	\arrow["\kappa", from=7-2, to=7-3]
	\arrow["\kappa", from=7-1, to=7-2]
	\arrow["\kappa", from=8-1, to=8-2]
	\arrow["{\kos \otimes 1}", from=7-1, to=8-1]
	\arrow["{\kos \otimes 1}", from=7-2, to=8-2]
	\arrow["{\kos \otimes 1}", from=8-1, to=9-1]
	\arrow["{p_\infty}"', from=5-4, to=5-5]
	\arrow["{p_\infty}"', from=4-4, to=4-5]
	\arrow["{p_\infty}"', from=2-4, to=2-5]
	\arrow[from=2-3, to=2-4]
	\arrow[from=2-3, to=3-3]
	\arrow["{\kos \otimes 1}", from=2-4, to=3-4]
	\arrow["{\kos \otimes 1}", from=2-5, to=3-5]
	\arrow["{\kos \otimes 1}", from=6-1, to=7-1]
	\arrow["{\kos \otimes 1}", from=6-2, to=7-2]
	\arrow["{\kos \otimes 1}", from=6-3, to=7-3]
	\arrow[from=1-5, to=2-5]
	\arrow["{\kos \otimes 1}", from=4-5, to=5-5]
	\arrow["{\kos \otimes 1}", from=4-4, to=5-4]
	\arrow["{p_\infty}", curve={height=24pt}, from=8-2, to=6-5]
	\arrow["{p_\infty}"', curve={height=30pt}, from=9-1, to=6-5]
	\arrow["{p_\infty}"', from=6-4, to=6-5]
	\arrow[from=6-3, to=6-4]
	\arrow[from=5-3, to=5-4]
	\arrow["\psi", from=5-5, to=6-5]
	\arrow["{\kos \otimes 1}", from=5-4, to=6-4]
\end{tikzcd}\]
\caption{The perturbed double complex.}\label{fig:perturbed}
\end{figure}   
    \item The perturbed deformation retract comes from the double complex of Figure \ref{fig:perturbed}, where the perturbation is precisely the vertical Koszul differential $\kos \otimes 1$ and
    \begingroup\allowdisplaybreaks
    \begin{align*}
        i_\infty &= \begin{cases}
    (1-h(\kos \otimes 1))^{-1} h & \textrm{on} \ L^i_{a,w} \\
    \varepsilon^{-1} & \textrm{on} \ R, \\
    \end{cases} \\
    p_\infty &= \begin{cases}
    \kappa & \textrm{on} \ (\bigwedge^i \otimes S_{a-1})_w, \ i>0, \\
    \varepsilon & \textrm{on} \ R, \\
    0 & \textrm{otherwise}. \\
    \end{cases}
    \end{align*}
    \endgroup
    The fact that the perturbed differential induced on the rightmost column of Figure \ref{fig:perturbed} is precisely $\kos \otimes 1$ is an identical computation to that done in $4.6$ of \cite{miller2020transferring}.
\end{enumerate}
\end{chunk}

\begin{theorem}\label{thm:theAlgStruct}
Adopt notation and hypotheses as in Setup \ref{set:thewLcomps}. The complex $L^w (\psi , d)$ of Definition \ref{def:theLwComp} admits the structure of an associative DG-algebra. If $R$ is a standard graded polynomial ring and $\im \psi = R_+$, then this product is invariant under the action of the symmetric group on the variables.
\end{theorem}

\begin{proof}
Define the product of the classes of $f_\sigma \otimes f^\alpha$ and $f_\tau \otimes f^\beta \in L^w (\psi , d)$ via
$$(f_\sigma \otimes f^\alpha)\cdot (f_\tau \otimes f^\beta) := p_{\infty} \big( i_{\infty} (f_\sigma \otimes f^\alpha) \cdot i_\infty (f_\tau \otimes f^\beta) \big).$$
This product will yield an associative DG-algebra structure on $L^w (\psi , d)$ by Proposition \ref{prop:perturbForDG} combined with Proposition \ref{prop:totalCxDG}, Lemma \ref{lem:scaleDerham}, Proposition \ref{prop:genLeibniz}, and the discussion of \ref{chunk:discussionOfTransfer}.
\end{proof}

It is not difficult to see that in the case $R = k[x_1 , \dots , x_n]$ and $\im (\psi ) = (x_1 , \dots , x_n)$ with $w = (1, 1 , \dots , 1)$, the complex $L^w (\psi , d)$ is identical to the complex constructed by Galetto in \cite{galetto2016ideal}. Thus, Theorem \ref{thm:theAlgStruct} gives an explicit $S_n$-invariant algebra structure on those complexes.

\begin{remark}
One can alternatively remove the characteristic assumption on $k$ at the expense of losing $S_n$-invariance by restricting the product of Srinivasan \cite{srinivasan1989algebra} to all standard tableaux with bounded multidegree. It is straightforward to verify that this restriction does yield a subalgebra.
\end{remark}

\bibliographystyle{amsplain}
\bibliography{biblio}

\providecommand{\bysame}{\leavevmode\hbox to3em{\hrulefill}\thinspace}
\providecommand{\MR}{\relax\ifhmode\unskip\space\fi MR }
\providecommand{\MRhref}[2]{%
  \href{http://www.ams.org/mathscinet-getitem?mr=#1}{#2}
}
\providecommand{\href}[2]{#2}
\begin{thebibliography}{10}

\bibitem{almousa2021}
Ayah Almousa and Keller VandeBogert, \emph{Polarizations and hook partitions},
  In Preparation.

\bibitem{aramova1998squarefree}
Annetta Aramova, J{\"u}rgen Herzog, and Takayuki Hibi, \emph{Squarefree
  lexsegment ideals}, Mathematische Zeitschrift \textbf{228} (1998), no.~2,
  353--378.

\bibitem{avramov1998infinite}
Luchezar~L Avramov, \emph{Infinite free resolutions}, Six lectures on
  commutative algebra, Springer, 1998, pp.~1--118.

\bibitem{buchsbaum1975generic}
David~A Buchsbaum and David Eisenbud, \emph{Generic free resolutions and a
  family of generically perfect ideals}, Advances in Mathematics \textbf{18}
  (1975), no.~3, 245--301.

\bibitem{el2014artinian}
Sabine El~Khoury and Andrew~R Kustin, \emph{Artinian gorenstein algebras with
  linear resolutions}, Journal of Algebra \textbf{420} (2014), 402--474.

\bibitem{galetto2016ideal}
Federico Galetto, \emph{On the ideal generated by all squarefree monomials of a
  given degree}, arXiv preprint arXiv:1609.06396 (2016).

\bibitem{gasharov2002resolutions}
Vesselin Gasharov, Takayuki Hibi, and Irena Peeva, \emph{Resolutions of
  a-stable ideals}, Journal of Algebra \textbf{254} (2002), no.~2, 375--394.

\bibitem{herzog2018koszul}
J{\"u}rgen Herzog and Rasoul~Ahangari Maleki, \emph{Koszul cycles and golod
  rings}, Manuscripta Mathematica \textbf{157} (2018), no.~3, 483--495.

\bibitem{katthan2017non}
Lukas Katth{\"a}n, \emph{A non-golod ring with a trivial product on its koszul
  homology}, Journal of Algebra \textbf{479} (2017), 244--262.

\bibitem{miller2020transferring}
Claudia Miller and Hamidreza Rahmati, \emph{Transferring algebra structures on
  complexes}, arXiv preprint arXiv:2007.08040 (2020).

\bibitem{raicu2021regularity}
Claudiu Raicu, \emph{Regularity of sn-invariant monomial ideals}, Journal of
  Combinatorial Theory, Series A \textbf{177} (2021), 105307.

\bibitem{srinivasan1989algebra}
Hema Srinivasan, \emph{Algebra structures on some canonical resolutions},
  Journal of Algebra \textbf{122} (1989), no.~1, 150--187.

\bibitem{weyman2003}
Jerzy Weyman, \emph{Cohomology of vector bundles and syzygies}, vol. 149,
  Cambridge University Press, 2003.

\end{thebibliography}
\addcontentsline{toc}{section}{Bibliography}

\end{document}